\documentclass[11]{amsart}

\usepackage{amsmath, amstext, amsgen, amsbsy, amsopn, amsfonts, amssymb, graphicx,psfrag}
\usepackage{amsthm}
\usepackage{comment}
\usepackage[normalem]{ulem}
\usepackage{pdfsync}
\usepackage{epstopdf}
\usepackage{color}
\usepackage[thinlines]{easytable}
\usepackage{multirow}
\usepackage{multicol}
\usepackage{hyperref}

\def\e{{\varepsilon}}

\newtheorem{theorem}{Theorem}[section]
\newtheorem{lemma}[theorem]{Lemma}
\newtheorem{corollary}[theorem]{Corollary}
\newtheorem{question}{Question}
\theoremstyle{definition}
\newtheorem*{definition}{Definition}

\theoremstyle{remark}

\newtheorem{example}[theorem]{Example}
\begin{document}

%\begin{frontmatter}

%\markboth{Mellor and Smith}{$N$-quandles of links}

\title{Classifying links and spatial graphs with finite $N$-quandles}

\author{Blake Mellor}
\address{Loyola Marymount University, 1 LMU Drive, Los Angeles, CA 90045}
%\ead{blake.mellor@lmu.edu}
\email{blake.mellor@lmu.edu}

\date{}
\maketitle
\begin{abstract} 
The fundamental quandle is a complete invariant for unoriented tame knots \cite{JO, Ma} and non-split links \cite{FR}. We prove a relationship between the components of the fundamental quandle and the cosets of the peripheral subgroup(s) in the fundamental group of the knot or link. We extend these relationships to spatial graphs, and to $N$-quandles of links and spatial graphs. As an application, we are able to give a complete list of links with finite $N$-quandles, proving a conjecture from \cite{MS}, and a partial list of spatial graphs with finite $N$-quandles.
\end{abstract}

%\begin{keyword}
%fundamental quandle, $N$-quandle, $n$-quandle.
%\MSC[2020] 57K12, 57M15
%\end{keyword}

%\end{frontmatter}

\section{Introduction}
The fundamental quandle of a link $L$ is an algebraic object that encodes the three Reidemeister moves. For unoriented tame knots, Joyce \cite{JO, JO2} and Matveev \cite{Ma} showed that the fundamental quandle is a complete invariant. The proof involves defining a quandle structure on the cosets of the peripheral subgroup of the knot in the fundamental group, and then proving that the resulting quandle is isomorphic to the fundamental quandle. A similar argument proves that the fundamental quandle is a complete invariant for unoriented, tame, non-split links.

While the fundamental quandle of a knot is a powerful invariant, it is often difficult to compute or to compare. Joyce \cite{JO, JO2} also introduced the fundamental $n$-quandle of a knot, which can be thought of as a quotient of the fundamental quandle where every element of the quandle has ``order'' $n$. These are simpler than the fundamental quandle, and in some cases are even finite. Hoste and Shanahan \cite{HS2} extended the relationship between the fundamental quandle and cosets in the fundamental group to $n$-quandles, and used it to give a complete list of links with finite $n$-quandles. In the current paper, we will further generalize this relationship in two ways: first, by extending it to a broader class of quotients of the fundamental quandle, called $N$-quandles \cite{MS}, and second by extending it to spatial graphs. In each case, we will show the fundamental quandle (or $N$-quandle) is isomorphic to a quandle defined on the cosets of a particular subgroup of the fundamental group (or a quotient of the fundamental group). As an application, we are able to give a complete list of links with finite $N$-quandles (verifying the conjecture in \cite{MS}), and, for graphs which are the singular locus of a three-dimensional orbifold, the list of graphs with finite $N$-quandles (proving part of a conjecture in \cite{BM}).

In Section \ref{S:quandles} we review the definitions of quandles and $N$-quandles, ending with a proof that the conjugation group of a finite $N$-quandle is also finite. In Section \ref{S:fundamental} we introduce the fundamental quandles for links and spatial graphs, and provide topological interpretations of these quandles.  The topological interpretation for knots and links is due to Fenn and Rourke \cite{FR}; we introduce an extension to spatial graphs. Then, in Section \ref{S:cosets}, we prove our main results about relationships between the fundamental quandle (and $N$-quandle) and the fundamental group. We apply these results in Section \ref{S:finite} to classify links and spatial graphs with finite $N$-quandles (a list is provided in the Appendix). Finally, in Section \ref{S:questions} we pose some questions for further investigation.

%%%%%%%%%%%%%%%%%%%%%%%%%%%%%%%%%%%%%%%%%%
%  Section: Quandles and $N$-quandles \label{S:quandles} 
%%%%%%%%%%%%%%%%%%%%%%%%%%%%%%%%%%%%%%%%%%
\section{Quandles and $N$-quandles} \label{S:quandles}

\subsection{Definitions and notation} \label{SS:definitions}

We begin with the definition of a quandle. We refer the reader to \cite{FR}, \cite{JO}, \cite{JO2}, and \cite{WI} for more detailed information.

A {\it  quandle} is a set $Q$ equipped with two binary operations $\rhd$ and $\rhd^{-1}$ that satisfy the following three axioms:
\begin{itemize}
\item[\bf A1.] $x \rhd x =x$ for all $x \in Q$.
\item[\bf A2.] $(x \rhd y) \rhd^{-1} y = x = (x \rhd^{-1} y) \rhd y$ for all $x, y \in Q$.
\item[\bf A3.] $(x \rhd y) \rhd z = (x \rhd z) \rhd (y \rhd z)$ for all $x,y,z \in Q$.
\end{itemize}

The operation $\rhd$ is, in general, not associative. To clarify the distinction between $(x \rhd y) \rhd z$ and $x \rhd (y \rhd z)$, we adopt the exponential notation introduced by Fenn and Rourke in \cite{FR} and denote $x \rhd y$ as $x^y$ and $x \rhd^{-1} y$ as $x^{\bar y}$. With this notation, $x^{yz}$ will be taken to mean $(x^y)^z=(x \rhd y)\rhd z$ whereas $x^{y^z}$ will mean $x\rhd (y \rhd z)$. We also use $x^{y^n}$ (where $n$ is a positive integer) to denote $x^{yy\cdots y}$, with $n$ copies of $y$ in the exponent.

The following useful lemma from \cite{FR} describes how to re-associate a product in a quandle. 

\begin{lemma} \label{leftassoc}
If $x, y, u$ and $v$ are elements of a quandle, then
$$\left(x^u \right)^{\left(y^v \right)}=x^{u \bar v y v} \ \ \ \ \mbox{and}\ \ \ \ \left(x^u \right)^{\overline{\left(y^v \right)}}=x^{u \bar v \bar y v}.$$
\end{lemma}

Using Lemma~\ref{leftassoc}, elements in a quandle given by a presentation $\langle S \mid R \rangle$ (where $S$ is a set of generators, and $R$ is a set of relations among the generators) can be represented as equivalence classes of expressions of the form $x^w$ where $x$ is a generator in $S$ and $w$ is a word in the free group on $S$ (with $\bar x$ representing the inverse of $x$).

%If $n$ is a natural number, a quandle $Q$ is an {\em $n$-quandle} if $x^{y^n} =x$ for all $x,y \in Q$, where by $y^n$ we mean $y$ repeated $n$ times. The notion of an $n$-quandle was introduced by Joyce \cite{JO, JO2}. Given a presentation $\langle S \,|\, R\rangle$ of $Q$, a presentation of  the quotient $n$-quandle $Q_n$ is obtained by adding the relations $x^{y^n}=x$ for every pair of distinct generators $x$ and $y$. 

%The action of the inner automorphism group Inn$(Q)$ on the quandle $Q$ decomposes the quandle into disjoint orbits. These orbits are the {\em components} (or {\em algebraic components}) of the quandle $Q$; a quandle is {\em connected} if it has only one component. We use these components to define an $N$-quandle:

Two elements $p$ and $q$ of quandle $q$ are in the same {\em component} (or {\em algebraic component}) if $p^w = q$ for some word $w$ in the free group on $Q$. This is an equivalence relations, so the components give a partition of the quandle. A quandle is {\em connected} if it has only one component.

\begin{definition} \label{D:Nquandle}
Given a quandle $Q$ with $k$ ordered components, labeled from 1 to $k$, and a $k$-tuple of natural numbers $N = (n_1, \dots, n_k)$, we say $Q$ is an {\em $N$-quandle} if $x^{y^{n_i}} = x$ whenever $x \in Q$ and $y$ is in the $i$th component of $Q$. 
\end{definition}

In the special case when $n_1=n_2=\cdots = n_k = n$, we have the $n$-quandle introduced by Joyce \cite{JO, JO2}. Note that the ordering of the components in an $N$-quandle is very important; the relations depend intrinsically on knowing which component is associated with which number $n_i$. 

Given a presentation $\langle S \,|\, R\rangle$ of $Q$, a presentation of  the quotient $N$-quandle $Q_N$ is obtained by adding the relations $x^{y^{n_i}}=x$ for every pair of distinct generators $x$ and $y$, where $y$ is in the $i$th component of $Q$.  

%%%%%%%%%%%%%%%%%%%%%%%%%%%%%%%%%
%  SubSection: The conjugation group of a quandle \label{SS:conj}
%%%%%%%%%%%%%%%%%%%%%%%%%%%%%%%%%
\subsection{The conjugation group of a quandle} \label{SS:conj}

In this section we will explain how to naturally associate a group to any quandle or $N$-quandle. For our purposes, it is enough to consider finitely presented quandles.  Suppose a quandle $Q$ has a presentation
$$Q = \langle q_1, \dots, q_s \mid r_1, \dots, r_m \rangle,$$
where each relation $r_i$ has the form $q_{a_i}^{w_i} = q_{b_i}$, with $a_i, b_i \in \{1, \dots, s\}$ and $w_i$ a word in the $q_j$'s and $\overline{q_j}$'s. Then the conjugation group has the presentation
$$Conj(Q) = \langle q_1, \dots, q_s \mid \overline{r_1}, \dots, \overline{r_m} \rangle,$$
where, for any quandle relation $r$ of the form $x^w = y$, $\overline{r}$ is the group relation $w^{-1}xwy^{-1} = 1$ (in the word $w$, $\overline{q_i}$ is interpreted as $q_i^{-1}$). In other words, the quandle operation is replaced by conjugation in the group. This group (under the name {\em Adconj}) was first defined by Joyce \cite{JO, JO2}.

If $Q$ is an $N$-quandle for a $k$-tuple $N = (n_1, \dots, n_k)$ (so $Q$ has $k$ components, denoted $Q_1, \dots, Q_k$), then we can also associate with $Q$ a natural quotient of the conjugation group, denoted $Conj_N(Q)$. Suppose generator $q_i$ is an element of $Q_{j_i}$, then:
$$Conj_N(Q) = \langle q_1, \dots, q_s \mid \overline{r_1}, \dots, \overline{r_m}, q_1^{n_{j_1}}, \dots, q_s^{n_{j_s}} \rangle = Conj(Q)/\langle q_1^{n_{j_1}}, \dots, q_s^{n_{j_s}} \rangle,$$ 
where $\langle q_1^{n_{j_1}}, \dots, q_s^{n_{j_s}} \rangle$ is the normal subgroup generated by $q_1^{n_{j_1}}, \dots, q_s^{n_{j_s}}$.

Our main result in this section is that, if $Q$ is a finite $N$-quandle, then $Conj_N(Q)$ is a finite group. In the special case when $Q$ is an $n$-quandle, this was proved by Joyce in his dissertation \cite{JO2}; our proof is a modified version.

\begin{theorem} \label{T:QtoG}
If $Q$ is a finite $(n_1,\dots,n_k)$-quandle with algebraic components $Q_1, \dots, Q_k$, then $Conj_N(Q)$ is a finite group, and $\vert Conj_N(Q) \vert \leq n_1^{\vert Q_1\vert} \cdots n_k^{\vert Q_k\vert}$.
\end{theorem}
\begin{proof}
Let $x_1, x_2, \dots, x_{\vert Q\vert}$ denote the elements of $Q$ and (abusing notation) also the corresponding elements of $Conj_N(Q)$. Suppose $x_i \in Q_j$; then there is a generator $q$ in $Q_j$ such that $q^w = x_i$ in $Q$.  In the group $Conj_N(Q)$, we have the relation $q^{n_j} = 1$, and $w^{-1}qw = x_i$.  Then 
$$x_i^{n_j} = (w^{-1}qw)^{n_j} = w^{-1}q^{n_j}w = w^{-1}w = 1.$$
So each element $x_i$ of $Q$ corresponds to an element of finite order in $Conj_N(Q)$, with the order determined by the algebraic component of $Q$ containing $x_i$.

We will prove inductively that any element $z$ of $Conj_N(Q)$ can be written as a product $z = x_1^{a_1}\cdots x_{\vert Q\vert}^{a_{\vert Q\vert}}$. If $x_i$ is in $Q_j$, then $0 \leq a_i < n_j$, so the number of such products is at most $n_1^{\vert Q_1\vert} \cdots n_k^{\vert Q_k\vert}$, giving the desired bound.

Since the generators of $Conj_N(Q)$ correspond to elements of $Q$, every element of $Conj_N(Q)$ can be written as some word in the $x_i$'s and $x_i^{-1}$'s; we will induct on the minimal length of these words. Certainly, if an element $z$ can be written as a single $x_i$ or $x_i^{-1}$, then we're done (note that $x_i^{-1} = x_i^{n_j - 1}$ for some $n_j$).

Now suppose that any element that can be written as a product of $m$ $x_i^{\pm 1}$'s can be rewritten as a product with the subscripts in non-decreasing order from left to right, {\em without} increasing the length of the product. Suppose $z$ has a minimal length of $m+1$ as a product of $x_i^{\pm 1}$'s. Then $z = x_j^\epsilon w$ for some $x_j$ and some word $w$ of length $m$ ($\epsilon = \pm 1$).  By our inductive hypothesis, $w$ can be rewritten with the subscripts in non-decreasing order, and still have length at most $m$. Now $w = x_l^\delta w'$ for some $x_l$, so $z = x_j^\epsilon x_l^\delta w'$.

If $j \leq l$, then $z$ is now a product with subscripts in non-decreasing order, and we're done. So suppose $l < j$.  In the quandle $Q$, $x_j \rhd^\delta x_l = x_t$ for some $t$. In the group $Conj_N(Q)$, this corresponds to a relation $x_l^{-\delta}x_j x_l^{\delta} = x_t$.  Hence $x_l^{-\delta}x_j^\epsilon x_l^{\delta} = x_t^\epsilon$, and so $x_j^\epsilon x_l^{\delta} = x_l^{\delta}x_t^\epsilon$. So we can rewrite $z = x_l^{\delta}x_t^\epsilon w'$, where $l < j$. But now $x_t^\epsilon w'$ is a word of length at most $m$, so it can be rewritten (without increasing its length) so that the subscripts are in non-decreasing order. We can repeat this process, each time reducing the subscript of the first factor of $z$. The process will eventually terminate with all subscripts in non-decreasing order (with the first factor as $x_1^\epsilon$, if not sooner).

So, by induction, every element $z$ can be written as a product of $x_i$'s and $x_i^{-1}$'s with the subscripts in non-decreasing order from left to right, and hence as a product $x_1^{a_1}\cdots x_{\vert Q\vert}^{a_{\vert Q\vert}}$.
\end{proof}

%%%%%%%%%%%%%%%%%%%%%%%%%%%%%%%%%
%  Section: Fundamental quandles of Links and Spatial Graphs \label{S:fundamental}
%%%%%%%%%%%%%%%%%%%%%%%%%%%%%%%%%
\section{Fundamental quandles of Links and Spatial Graphs} \label{S:fundamental}

\subsection{Wirtinger presentations for fundamental quandles} \label{SS:wirtinger}

If $\Gamma$ is an oriented knot, link or spatial graph in $\mathbb{S}^3$, then a presentation of its fundamental quandle, $Q(\Gamma)$, can be derived from a regular diagram $D$ of $\Gamma$ by a process similar to the Wirtinger algorithm. This was developed for links by Joyce \cite{JO}, and extended to spatial graphs by Niebrzydowski \cite{Ni}. We assign a quandle generator $x_1, x_2, \dots , x_n$ to each arc of $D$ (or, if $\Gamma$ is a spatial graph, to each arc of an edge), then at each crossing introduce the relation $x_i=x_k^{ x_j}$ as shown on the left in Figure~\ref{relations}. For spatial graphs, at a vertex with incident edges $x_1, x_2, \dots x_n$, as shown on the right in Figure~\ref{relations}, we introduce the relation $y^{x_1^{\e_1}x_2^{\e_2}\cdots x_n^{\e_n}} = y$ (where $\e_i = 1$ if $a_i$ is directed into the vertex, and $\e_i = -1$ if $a_i$ is directed out from the vertex). Here $y$ can be {\em any} element of the quandle; for a finite presentation it suffices to consider the cases when $y$ is a generator of the quandle. It is easy to check that the Reidemeister moves for links and spatial graphs do not change the quandle given by this presentation so that the quandle is indeed an invariant of $\Gamma$.

\begin{figure}[h]
$$\includegraphics[height=1in]{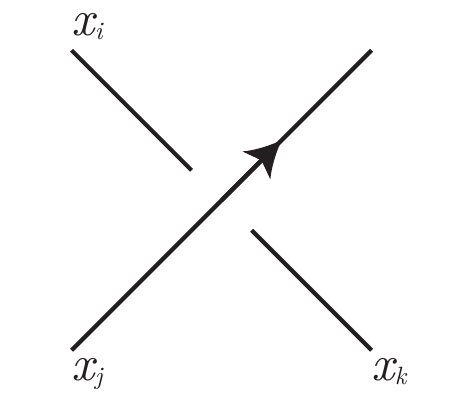} \qquad \qquad \includegraphics[height = 1.1in]{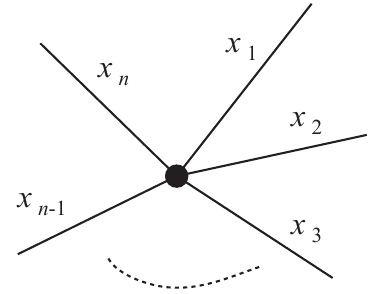}$$
$$\qquad  x_i=x_k^{x_j} \qquad\qquad\qquad \qquad y^{x_1^{\e_1}x_2^{\e_2}\cdots x_n^{\e_n}} = y$$
\caption{The fundamental quandle relations at a crossing and at a vertex.}
\label{relations}
\end{figure}

Fenn and Rourke \cite{FR} observed that, for a link $L$, the components of the quandle $Q(L)$ are in bijective correspondence with the components of $L$, with each component of the quandle containing the generators of the Wirtinger presentation associated to the corresponding link component. Similarly, for a spatial graph $G$, the components of the quandle $Q(G)$ correspond to the edges of the graph \cite{BM}.  

\begin{definition}
Suppose $\Gamma$ is a link (resp. graph) with $k$ components (resp. edges), such that each component (resp. edge) $c_i$ is labeled with a natural number $n_i$, and let $N = (n_1, \dots, n_k)$.  If the fundamental quandle $Q(\Gamma)$ has the Wirtinger presentation from a diagram $D$, then $Q_N(\Gamma)$ is the quotient obtained by adding relations $x^{y^{n_i}} = x$ for each pair of distinct generators $x$ and $y$ where $y$ corresponds to an arc of component (or edge) $c_i$ in the diagram $D$.  $Q_N(\Gamma)$ is called the {\em fundamental $N$-quandle of the link or graph} (and depends on the ordering of the link components/edges).
\end{definition}

If $L$ is a link, then the groups $Conj(Q(L))$ and $Conj_N(Q_N(L))$ have natural interpretations. From the Wirtinger presentation description of the fundamental quandle, it is immediate that $Conj(Q(L))$ is the fundamental group $\pi_1(\mathbb{S}^3 - L)$. If we select a meridian $\mu_i$ for each component, then $Conj_N(Q_N(L)) = \pi_1(\mathbb{S}^3 - L)/\langle \mu_i^{n_i} \rangle$, where $\langle \mu_i^{n_i} \rangle$ is the normal subgroup generated by $\{\mu_i^{n_i}\}$.

If $G$ is a spatial graph, then the quandle relation at the $i$th vertex has the form $x^{w_i} = x$.  In $Conj(Q(G))$, this becomes $w_i^{-1}xw_ix^{-1} = 1$, which is a weaker relation than the corresponding vertex relation in $\pi_1(\mathbb{S}^3 - G)$, namely $w_i = 1$.  So in this case, we have that $Conj(Q(G))/\langle w_i \rangle = \pi_1(\mathbb{S}^3 - G)$, where $\langle w_i\rangle$ is the normal subgroup generated by $\{w_i\}$.  Similarly, if we select a meridian $\mu_i$ for each edge of the graph, then $Conj_N(Q_N(G))/\langle w_i \rangle = \pi_1(\mathbb{S}^3 - G)/\langle \mu_i^{n_i} \rangle$, where $\langle \mu_i^{n_i} \rangle$ is the normal subgroup generated by $\{\mu_i^{n_i}\}$.

This gives us the following Corollary to Theorem \ref{T:QtoG}:

\begin{corollary} \label{C:QtoG}
If $\Gamma$ is a link or spatial graph, and $Q_N(\Gamma)$ is a finite quandle, then $\pi_1(\mathbb{S}^3 - \Gamma)/\langle \mu_i^{n_i} \rangle$ is a finite group.
\end{corollary}

In Section \ref{S:cosets} we will prove the converse of Corollary \ref{C:QtoG}.

%%%%%%%%%%%%%%%%%%%%%%%%%%%%%%%%%
%  subSection: topological interpretation of fundamental quandle
%%%%%%%%%%%%%%%%%%%%%%%%%%%%%%%%%

\subsection{A topological interpretation of the fundamental quandle} \label{SS:topological}

Fenn and Rourke \cite{FR} provided a topological interpretation for the fundamental quandle of a knot, and Hoste and Shanahan \cite{HS2} extended it to $n$-quandles of links. In this section, we review this interpretation, and extend it to first to spatial graphs, and then to $N$-quandles for both links and spatial graphs.

For a link $L$, let $X = \mathbb{S}^3 - N(L)$ be the exterior of the link, and choose a basepoint (denoted $*$) in $X$.  Then $T(L)$ is defined to be the homotopy classes of paths $\alpha: [0,1] \rightarrow X$ such that $\alpha(0) = *$ and $\alpha(1) \in \partial X$.  Moreover, the homotopies must be through sequences of paths with one endpoint at $*$ and the other on $\partial X$. We define quandle operations $\rhd$ and $\rhd^{-1}$ on $T(L)$ by
$$\alpha \rhd^{\pm 1} \beta = \beta m_{\beta}^{\mp 1} \beta^{-1} \alpha$$
where $m_\beta$ is a meridian of $L$ that begins and ends at $\beta(1)$. In other words, $m_\beta$ is a loop in $\partial N(L)$ which is essential in $\partial N(L)$, null-homotopic in $N(L)$, and has linking number 1 with $L$. So the path $\alpha \rhd \beta$ is formed by following $\beta$ from $b$ to $\partial X$, going around the meridian, following $\beta$ back to $*$, and then traversing the path $\alpha$ (see Figure \ref{F:multiply}). Observe that for each component $L_i$ of the link $L$, the paths which have one endpoint on $\partial N(L_i)$ form an algebraic component of the quandle $T(L)$. Fenn and Rourke \cite[Theorem 4.7]{FR} proved that $Q(L)$ and $T(L)$ are isomorphic quandles.

\begin{figure}[h]
$$\scalebox{.8}{\includegraphics{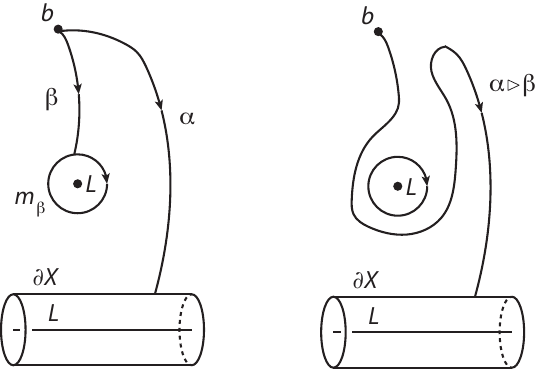}}$$
\caption{Multiplying paths $\alpha$ and $\beta$ in $T(L)$ to form $\alpha \rhd \beta$.}
\label{F:multiply}
\end{figure}

If $G$ is a spatial graph, we can decompose $N(G)$ into a union of balls (centered at each vertex of the graph) and solid cylinders $N(e_i) = D^2 \times [0,1]$ around each edge of the graph.  We choose these so that the portion of the graph inside each ball has a projection with no crossings, and so that the cylinders $N(e_i)$ are all disjoint. Then the meridians of the edge $e_i$ are (homotopic to) the loops $S^1 \times \{t\}$ in $\partial N(e_i)$. We define $T(G)$ in the same way as for links, except that we only consider paths with endpoints in $\partial X \cap \bigcup{\partial N(e_i)}$, and the homotopies are through paths with one endpoint at $*$ and the other in $\partial X \cap \bigcup{\partial N(e_i)}$.  In other words, the endpoints of the paths are allowed to wander around the boundary of the cylinder surrounding each edge, but are not allowed to be on the boundaries of the balls around each vertex, and hence cannot move between edges. Since the endpoint of a path is restricted to a single edge, and the edges have well-defined meridians, we can define the quandle operation for graphs in the same way as for links; the algebraic components of the quandle now correspond to the edges of the graph. 

\begin{theorem} \label{T:T=Qgraph}
For a spatial graph $G$, the quandles $T(G)$ and $Q(G)$ are isomorphic.
\end{theorem}
\begin{proof}
The proof that $T(G)$ and $Q(G)$ are isomorphic proceeds exactly as it does for links in \cite[Theorem 4.7]{FR}; the only modification is that we need to account for the vertices.  Namely, we need to check:
\begin{enumerate}
	\item In the map from $T(G) \rightarrow Q(G)$, that homotoping a path under a vertex does not change the resulting element of $Q(G)$. This is guaranteed by the vertex relations in $Q(G)$. 
	\item In the map from $Q(G) \rightarrow T(G)$, the paths resulting from an application of a vertex relation are homotopic. This is easily seen by the same approach used for the crossings.
\end{enumerate}

Those familiar with the argument in \cite[Theorem 4.7]{FR} may safely move on; for the convenience of other readers, and for later reference, we will include the details of the proof, including the modifications needed for spatial graphs.

We will define quandle homomorphisms $\lambda: T(G) \rightarrow Q(G)$ and $\mu: Q(G) \rightarrow T(G)$, and show that they are inverses.  

{\bf Definition of $\lambda$:} Suppose that $\gamma \in T(G)$ is a path from a point $p$ on the boundary of $N(a)$, where $a$ is an edge of the graph, to the basepoint $*$.  Consider a projection of graph in which $\gamma$ goes {\em under} arcs $b, c, d, \dots$ as it goes from $p$ to $*$, as shown in Figure~\ref{F:TtoQ}. Then we define $\lambda(\gamma) = a^{b^{\e}c^{\e}d^{\e}\dots}$, where $\e = \pm 1$, depending on whether the crossing with each arc has positive or negative sign.

\begin{figure}[ht]
$$\scalebox{.8}{\includegraphics{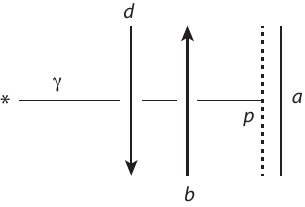}}$$
\caption{$\lambda(\gamma) = a^{b\bar{d}}.$}
\label{F:TtoQ}
\end{figure}

To show $\lambda$ is well-defined, we need to show that if $\gamma \tilde \gamma'$, then $\lambda(\gamma) = \lambda(\gamma')$ in $Q(G)$. There are a few cases to consider.  The first is a homotopy that pushes the endpoint $p$ from an arc $a$ to another arc $c$ in the projection of an edge, as shown in Figure~\ref{F:TtoQ_1}. Since $c = a^b$, $c^{\bar{b}w} = a^{b\bar{b}w} = a^w = \lambda(\gamma)$.

\begin{figure}[ht]
$$\scalebox{.8}{\includegraphics{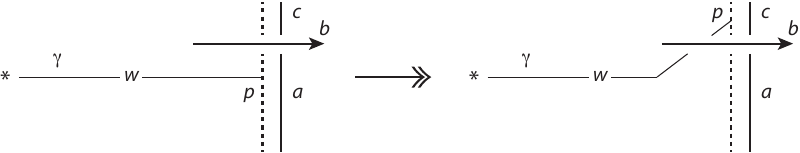}}$$
\caption{Homotopy pushing $p$ to a neighboring arc of the edge.}
\label{F:TtoQ_1}
\end{figure}

The second case is similar to a type II Reidemeister move, shown in Figure~\ref{F:TtoQ_2}, where $\gamma$ is pushed off an arc of $G$. Here $\lambda(\gamma) = a^{w\bar{e}ez} = a^{wz}$.

\begin{figure}[ht]
$$\scalebox{.8}{\includegraphics{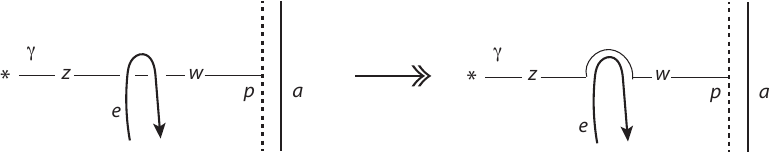}}$$
\caption{Homotopy pushing $\gamma$ off an arc of $G$.}
\label{F:TtoQ_2}
\end{figure}

The third case is when the path $\gamma$ is pushed under a crossing of two edges of $G$, as shown in Figure~\ref{F:TtoQ_3}. In this case $d = c^b$, so $\lambda(\gamma) = a^{wd\bar{b}z} = a^{w\bar{b}cb\bar{b}z} = a^{w\bar{b}cz}$.

\begin{figure}[ht]
$$\scalebox{.8}{\includegraphics{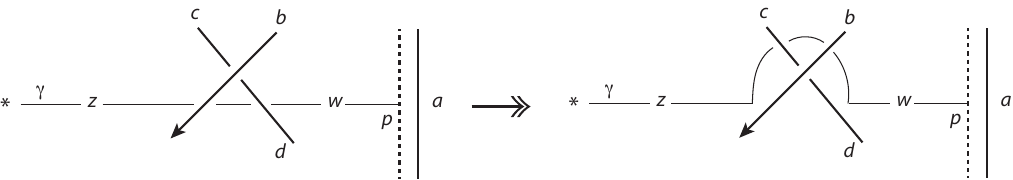}}$$
\caption{Homotopy pushing $\gamma$ across a crossing of $G$.}
\label{F:TtoQ_3}
\end{figure}

Finally, for graphs, we also need to consider homotopies that push $\gamma$ across a vertex of $G$, as shown in Figure~\ref{F:TtoQ_4}. In this case, for any $y \in Q(G)$, $y^{b^\e c^\e d^\e} = y$, as in Figure \ref{relations}.  So $\lambda(\gamma) = a^{wb^\e c^\e d^\e z} = a^{wz}$.  So the map $\lambda$ is well-defined on $T(G)$.  Finally, from the definition of multiplication in $T(G)$, it is clear that if $\lambda(\alpha) = a^w$ and $\lambda(\beta) = b^z$, then $\lambda(\alpha \rhd \beta) = a^{w\bar{z}bz} = (a^w)^{b^z} = \lambda(\alpha) \rhd \lambda(\beta)$, so $\lambda$ is a quandle homomorphism.

\begin{figure}[ht]
$$\scalebox{.8}{\includegraphics{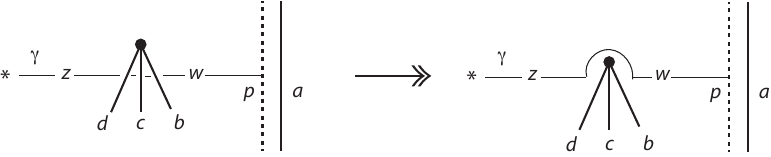}}$$
\caption{Homotopy pushing $\gamma$ across a vertex of $G$.}
\label{F:TtoQ_4}
\end{figure}

{\bf Definition of $\mu$:} Given a projection of the graph $G$, and an arc $a$ in the projection, we first define $\mu(a)$ as a path from the basepoint $*$ to a point on the boundary of $N(a)$ that goes {\em over} any other arc in the graph.  Given another arc $c$, we extend this to define $\mu(a^c)$ as the result of composing $\mu(a)$ with a path that goes from $*$ to a point on $N(c)$ (passing over any other arcs along the way), loops around $c$ once in the positive direction (negative for $a^{\bar{c}}$), and then returns to $*$ along the same path, as shown in Figure~\ref{F:QtoT}.

\begin{figure}[ht]
$$\scalebox{.8}{\includegraphics{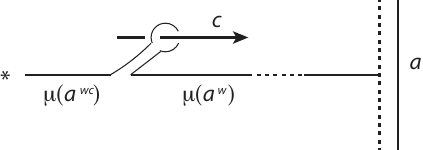}}$$
\caption{The path $\mu(a^c)$}
\label{F:QtoT}
\end{figure}

As Fenn and Rourke \cite{FR} observe, to show $\mu$ is well-defined it suffices to consider two types of substitutions.  

{\em Primary substitution:} replace $c^w$ by $a^{bw}$ when $c = a^b$. This is illustrated in Figure~\ref{F:QtoT_1}; the two paths are clearly homotopic by moving the endpoint $p$ under arc $b$.

\begin{figure}[ht]
$$\scalebox{.8}{\includegraphics{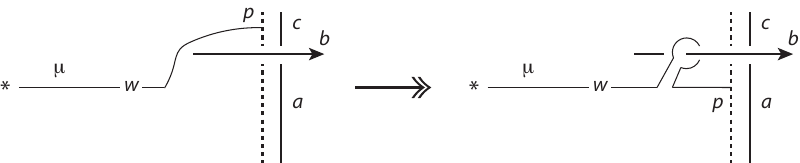}}$$
\caption{When $c = a^b$, $\mu(c^w) = \mu(a^{bw})$}
\label{F:QtoT_1}
\end{figure}

{\em Secondary substitution:} replace $x^{wz}$ by $x^{wvz}$, where $x^v = x$ for all $x$ in $Q(G)$.  In $Q(G)$, these relations are generated by crossing relations (if $c = a^b$, then $x^{\bar{b}abc} = x$ for all $x$) and by vertex relations. Figure~\ref{F:QtoT_2} shows that if $d = c^b$, then $\mu(a^{wz}) = \mu(a^{w\bar{b}cbdz})$. In this case the homotopy consists of pulling the strands of the path over the arcs of the crossing to get a loop that goes entirely under the crossing, and then contracting this loop back to the original path.

\begin{figure}[ht]
$$\scalebox{.8}{\includegraphics{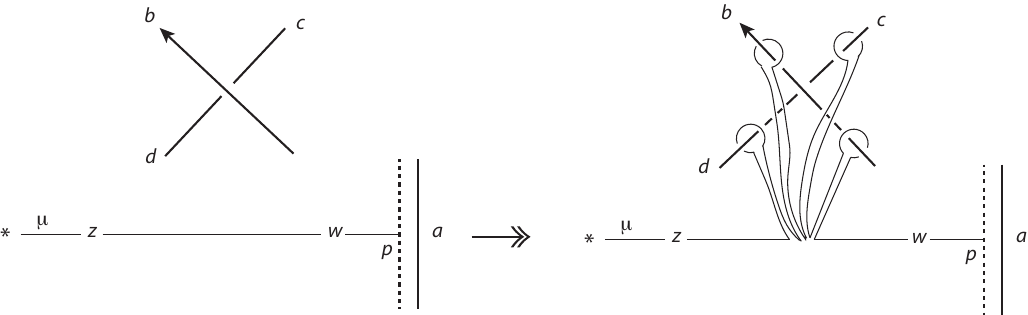}}$$
\caption{When $d = c^b$, $\mu(a^{wz}) = \mu(a^{w\bar{b}cbdz})$.}
\label{F:QtoT_2}
\end{figure}

Figure~\ref{F:QtoT_3} shows that at the vertex with relation $y^{b\bar{c}d} = y$, we have $\mu(a^{wz}) = \mu(a^{wb\bar{c}dz})$, by a homotopy similar to the one used for a crossing.

\begin{figure}[ht]
$$\scalebox{.8}{\includegraphics{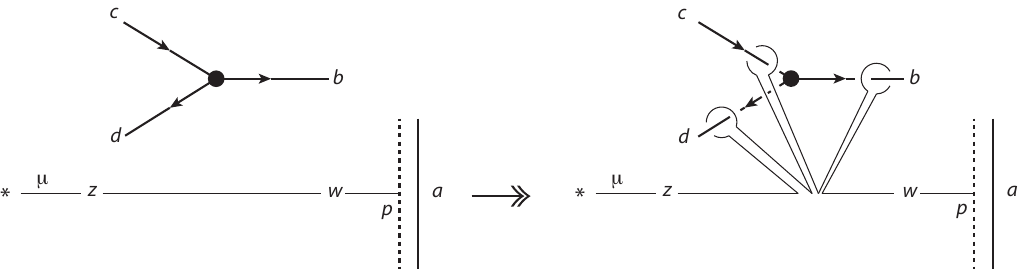}}$$
\caption{The path $\mu(a^c)$}
\label{F:QtoT_3}
\end{figure}

Hence, $\mu$ is well-defined, and it is immediate that $\mu(a^c) = \mu(a) \rhd \mu(c)$, so it is also a quandle homomorphism.

Finally, we observe that $\lambda \circ \mu$ is the identity on $Q(G)$, and $\mu \circ \lambda$ is the identity on $T(G)$ (since each undercrossing of the path with an arc of the graph can be homotoped to a loop around the arc, with a ``feeler" back to the basepoint $*$).  This completes the proof that $T(G)$ and $Q(G)$ are isomorphic quandles.
\end{proof}

To extend our topological interpretation to $N$-quandles, we generalize the $n$-meridian moves introduced in \cite{HS2}.

\begin{definition} \label{D:n-meridian}
Suppose $\Gamma$ is a link (resp. spatial graph) with $k$ components (resp. edges), and $N = (n_1, \dots, n_k)$. Let $c_i$ represent the $i$th component (resp. edge), and $m_i$ be a meridian of $c_i$. Suppose $\alpha$ is a path in $X$ with $\alpha(0) = *$ and $\alpha(1) \in \{*\} \cup \left(\bigcup_i{\partial N(c_i)}\right)$. Suppose further there is a $t_0 \in [0,1]$ such that $\alpha(t_0) \in \partial N(c_j)$.  Let $\sigma_1(t) = \alpha(tt_0)$ and $\sigma_2(t) = \alpha((1-t)t_0 + t)$, so $\alpha = \sigma_1\sigma_2$.  Then we say the path $\sigma_1m_j^{\pm n_j} \sigma_2$ is obtained from $\alpha$ by a $\pm N$-{\em meridian move}, as shown in Figure~\ref{F:N_meridian}.  Two paths are {\em N-meridionally equivalent} if they are related by a sequence of $\pm N$-meridian moves and homotopies.
\end{definition}

\begin{figure}[ht]
$$\scalebox{.8}{\includegraphics{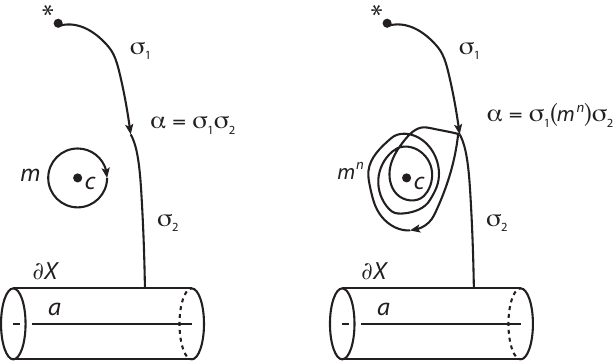}}$$
\caption{An $N$-meridian move. Here $m$ is the meridian for arc $c$ with label $n$.}
\label{F:N_meridian}
\end{figure}

We now define the $N$-quandle $T_N(\Gamma)$ as the set of $N$-meridional equivalence classes of paths in $T(\Gamma)$, with the same quandle operations as defined for $T(\Gamma)$. As before, the algebraic components of $T_N(\Gamma)$ are the sets of paths which end on the same $\partial N(c_i)$.

\begin{theorem} \label{T:T_N=Q_N}
The $N$-quandles $T_N(\Gamma)$ and $Q_N(\Gamma)$ are isomorphic.
\end{theorem}
\begin{proof}
Again, this closely follows the proof in \cite{FR}.  The addition of $N$-meridional equivalence among the paths in $T_N(\Gamma)$, as shown in Figure \ref{F:N_meridian} exactly corresponds to the addition of the relations $x^{y^{n_i}} = x$ in $Q_N(\Gamma)$ (where $y$ is in the $i$th algebraic component of $Q_N(\Gamma)$). To be precise, consider the maps $\lambda$ and $\mu$ from Theorem \ref{T:T=Qgraph}.  If $\lambda(\alpha) = \lambda(\sigma_1\sigma_2) = a^{wz}$, then $\lambda(\sigma_1m_j^{n_j} \sigma_2) = a^{wc_j^{n_j}z} = a^{wz}$ in $Q_N(\Gamma)$, so $\lambda$ is still well-defined on $T_N(\Gamma)$.  Conversely, $\mu(a^{wc_j^{n_j}z}) = \sigma_1m_j^{n_j} \sigma_2 = \sigma_1\sigma_2$ in $T_N(\Gamma)$, so $\mu$ is also well-defined on $Q_N(\Gamma)$.  The rest of the argument in Theorem \ref{T:T=Qgraph} is the same, proving $T_N(\Gamma)$ and $Q_N(\Gamma)$ are isomorphic.
\end{proof}

%%%%%%%%%%%%%%%%%%%%%%%%%%%%%%%%%
%  Section: Relating $Q_N(L)$ to cosets in $\pi_1(\mathbb{S}^3 - L)/\langle \mu_i^{n_i} \rangle$
%%%%%%%%%%%%%%%%%%%%%%%%%%%%%%%%%
\section{Relating $Q_N(\Gamma)$ to cosets in $\pi_1(\mathbb{S}^3 - \Gamma)/\langle \mu_i^{n_i} \rangle$} \label{S:cosets}

In \cite{JO}, Joyce defined a quandle structure on the set of cosets of the peripheral subgroup of the fundamental group of a knot $K$, and proved the resulting quandle is isomorphic to the knot quandle. This was a key part of his proof that the knot quandle classifies unoriented tame knots. Hoste and Shanahan \cite{HS2} extended this to the fundamental $n$-quandle of a link. Our goal in this section is to further extend to the fundamental $N$-quandles of links and spatial graphs.

Suppose $\Gamma$ is a link (resp. spatial graph) with $k$ components (resp. edges), and $N = (n_1, \dots, n_k)$. Let $\mu_i$ be a meridian for the $i$th component (resp. edge) in $\pi_1(\mathbb{S}^3-\Gamma)$. For convenience, let $\pi_1^N(\Gamma) = \pi_1(\mathbb{S}^3 - \Gamma)/\langle \mu_i^{n_i} \rangle$. We will define {\em peripheral subgroups} $P_i$ (for $1 \leq i \leq k$) as follows.  If $\Gamma$ is a link, and $c_i$ is the $i$th component, let $\lambda_i$ be a longitude for $c_i$, and define $P_i$ as the subgroup of $\pi_1^N(\Gamma)$ generated by $\mu_i$ and $\lambda_i$. If $\Gamma$ is a graph, and $c_i$ is the $i$th edge, then $P_i$ is just the (cyclic) subgroup of $\pi_1^N(\Gamma)$ generated by $\mu_i$. Note that in both cases, $P_i$ is an abelian group (since the meridian and longitude of a torus commute).

We denote the set of cosets of $P_i$ in $\pi_1^N(\Gamma)$ by $P_i\backslash \pi_1^N(\Gamma)$.  We define a quandle operation on the cosets by:
$$P_i g \rhd^{\pm 1} P_i h = P_i gh^{-1}\mu_i^{\pm 1} h$$
To see this operation is well-defined, suppose $P_i g = P_i r$ and $P_i h = P_i s$. So there are $p, q \in P_i$ such that $pg = r$ and $qh = s$.  Then
$$P_i r \rhd P_i s = P_i rs^{-1}\mu_i s = P_i pgh^{-1}q^{-1}\mu_i qh = P_i gh^{-1}(q^{-1}\mu_i q) h$$
Since $P_i$ is abelian, $q^{-1}\mu_i q = q^{-1}q\mu_i = \mu_i$, so $P_i r \rhd P_i s = P_i g \rhd P_i h$.  Hence the operation is well-defined, and it is straightforward to check that it satisfies the quandle axioms. We denote this quandle by $(P_i\backslash \pi_1^N(\Gamma); \mu_i)$. The following theorem was proved by Hoste and Shanahan \cite{HS2} for the fundamental $n$-quandle of a link; however, since the proof was done for each algebraic component, we can simply replace $n$ with $n_i$ to extend it to the fundamental $N$-quandle.

\begin{theorem} \label{T:GtoQ} \cite{HS2}
If $L = \{K_1, \dots, K_k\}$ is a link in $\mathbb{S}^3$, $N = (n_1, \dots, n_k)$ is a $k$-tuple of positive integers, and $P_i$ is the subgroup of $\pi_1^N(L)$ generated by a meridian $\mu_i$ and longitude $\lambda_i$ of $K_i$, then the quandle $(P_i\backslash \pi_1^N(L); \mu_i)$ is isomorphic to the $i$th algebraic component $Q_N^i(L)$ of $Q_N(L)$.
\end{theorem}

We will prove the corresponding theorem for spatial graphs.

\begin{theorem} \label{T:GtoQgraph}
If $G$ is a spatial graph with $k$ edges $e_1, \dots, e_k$, $N = (n_1, \dots, n_k)$ is a $k$-tuple of positive integers, and $P_i$ is the subgroup of $\pi_1^N(G)$ generated by a meridian $\mu_i$ of $e_i$, then the quandle $(P_i\backslash \pi_1^N(G); \mu_i)$ is isomorphic to the $i$th algebraic component $Q_N^i(G)$ of $Q_N(G)$. Similarly, if $P_i$ is the subgroup of $\pi_1(\mathbb{S}^3-G)$ generated by $\mu_i$, then the quandle $(P_i\backslash \pi_1(\mathbb{S}^3-G); \mu_i)$ is isomorphic to the $i$th algebraic component $Q^i(G)$ of $Q(G)$.
\end{theorem}
\begin{proof}
Without loss of generality, we will consider the subgroup $P_1$ generated by a meridian $\mu_1$ of the edge $e_1$. We begin by fixing an element $\nu \in Q_N^1(G)$ which is a path in $X = \mathbb{S}^3 - N(G)$ from the basepoint $b$ to a point in $\partial N(e_1)$.  Then $m_\nu$ is the meridian in $\partial N(e_1)$ that starts and ends at $\nu(1)$. So we let $\mu_1 = \nu m_\nu \nu^{-1} \in \pi_1^N(G)$, and $P_1$ is the subgroup generated by $\mu_1$. We define a map $\tau: (P_i\backslash \pi_1^N(G); \mu_i) \rightarrow Q_N(G)$ by $\tau(P_1 \alpha) = \alpha^{-1}\nu$. We need to show that $\tau$ is (1) well-defined, (2) onto $Q_N^1(G)$, (3) injective and (4) a quandle homomorphism.
\medskip

\noindent {\bf Well-defined.} Suppose $P_1\alpha = P_1\beta$, so $\beta = \mu_1^j \alpha$ for some $j$. Then 
$$\tau(P_i \beta) = \beta^{-1}\nu = \alpha^{-1} \mu_1^{-j} \nu = \alpha^{-1} (\nu m_\nu \nu^{-1})^{-j} \nu = \alpha^{-1} \nu m_\nu^{-j} \nu^{-1} \nu =  \alpha^{-1} \nu m_\nu^{-j}.$$
But since the endpoint of the path can move around $\partial N(e_1)$, $\nu m_\nu^{-j} \sim \nu$, so $\tau(P_i \beta) \sim \alpha^{-1} \nu = \tau(P_1 \alpha)$.  Hence $\tau$ is well-defined.
\medskip

\noindent {\bf Onto $Q_N^1(G)$.} For any $\alpha \in \pi_1^N(G)$, the endpoint $\alpha^{-1} \nu (1) = \nu(1) \in \partial N(e_1)$, so the image of $\tau$ is a subset of $Q_N^1(G)$. To show $\tau$ is onto $Q_N^1(G)$, consider $\sigma \in Q_N^1(G)$.  Then $\alpha = \nu \sigma^{-1} \in \pi_1^N(G)$, and $\tau(P_1 \alpha) = \sigma \nu^{-1} \nu = \sigma$. So the image of $\tau$ is equal to $Q_N^1(G)$.
\medskip

\noindent {\bf Injective.} Suppose $\tau(P_1\alpha) = \tau(P_1\beta)$.  Then $\alpha^{-1}\nu = \beta^{-1}\nu$, so $\alpha\beta^{-1}\nu = \nu$. In other words, there is a sequence of homotopies and $N$-meridian moves which transforms $\alpha\beta^{-1}\nu$ into $\nu$. During these homotopies, the endpoint of the path on $\partial N(e_1)$ traces out a loop from $\nu(1)$ back to $\nu(1)$, which is homotopic to $m_\nu^j$ for some $j$.  This means that, fixing both endpoints, we have $\alpha\beta^{-1}\nu m_\nu^j \sim \nu$.  Hence the loop $\alpha\beta^{-1}\nu m_\nu^j \nu^{-1}$ is trivial in $\pi_1^N(G)$.  This loop is the same as $\alpha\beta^{-1} \mu_1^j$, which means $\alpha\beta^{-1} = \mu_1^{-j} \in P_1$.  Hence $P_1 \alpha = P_1 \mu_1^{-j} \beta = P_1 \beta$, so $\tau$ is injective.
\medskip

\noindent {\bf Quandle homomorphism.} Consider $\alpha, \beta \in \pi_1^N(G)$.
\begin{align*}
\tau(P_1 \alpha \rhd P_1 \beta) &= \tau(P_1 \alpha \beta^{-1} \mu_1 \beta) \\
&= \beta^{-1} \mu_1^{-1} \beta \alpha^{-1} \nu \\
&= \beta^{-1} \nu m_\nu^{-1} \nu^{-1} \beta \alpha^{-1} \nu \\
&= (\beta^{-1} \nu) m_\nu^{-1} (\beta^{-1} \nu)^{-1} (\alpha^{-1} \nu) \\
&= (\alpha^{-1} \nu) \rhd (\beta^{-1} \nu) \\
&= \tau(\alpha) \rhd \tau(\beta)
\end{align*}

Therefore, $\tau$ is a quandle isomorphism between $(P_1\backslash \pi_1^N(G); \mu_1)$ and $Q_N^1(G)$. The same argument can be used for any $i$, $1 \leq i \leq k$. If we leave out the $N$-meridian moves, then the same proof shows $(P_i\backslash \pi_1(\mathbb{S}^3-G); \mu_i)$ is isomorphic to the $i$th algebraic component $Q^i(G)$ of $Q(G)$.
\end{proof}

If the group $\pi_1^N(\Gamma)$ is finite, then so is the set of cosets of $P_i$ for each $i$; hence the algebraic components of $Q_N(\Gamma)$ are finite as well.  Combining Theorems \ref{T:GtoQ} and \ref{T:GtoQgraph} with Corollary \ref{C:QtoG}, we conclude:

\begin{theorem} \label{T:QiffG}
For a link or spatial graph $\Gamma$, the fundamental $N$-quandle $Q_N(\Gamma)$ is finite if and only if the group $\pi_1^N(\Gamma)$ is finite.
\end{theorem}

In fact, we can make the relationship between the cardinalities of $Q_N(\Gamma)$ and $\pi_1^N(\Gamma)$ explicit.  Observe that, if $\Gamma$ is a spatial graph, then $\vert P_i\vert = \vert \langle \mu_i \rangle \vert = n_i$.

\begin{corollary} \label{C:sizes}
Let $\Gamma$ be any link or spatial graph, with $k$ components (or edges). If $\pi_1^N(\Gamma)$ and $Q_N(\Gamma)$ are finite, and $P_i$ is the peripheral subgroup for the $i$th component (or edge), then $\vert \pi_1^N(\Gamma)\vert = \vert P_i \vert \vert Q_N^i(\Gamma)\vert$, for any $1 \leq i \leq k$. In particular, if $\Gamma$ is a spatial graph, then $\vert \pi_1^N(\Gamma)\vert = n_i \vert Q_N^i(\Gamma)\vert$.
\end{corollary}

\begin{example}
Consider the knotted tetrahedron $G$ in Figure~\ref{F:K4knot}, with the labeling $N = (3, 3, 2, 2, 2, 2)$. As we will see, $Q_N(G)$ is finite, but attempts to compute it directly using {\it Mathematica} \cite{Me2} proved extremely lengthy.  However, $\pi_1^N(G)$ was computed very quickly using Miller's implementation of the Todd-Coxeter algorithm \cite{Mi}, and we found $\vert \pi_1^N(G) \vert = 2880$.  Hence $\vert Q_N^1(G)\vert = \vert Q_N^2(G)\vert = 2880/3 = 960$ and $\vert Q_N^3(G) \vert = \vert Q_N^4(G) \vert = \vert Q_N^5(G) \vert = \vert Q_N^6(G) \vert = 2880/2 = 1440$.  So $\vert Q_N(G) \vert = 960(2) + 1440(4) = 7680$.

\begin{figure}[h]
$$\scalebox{1}{\includegraphics{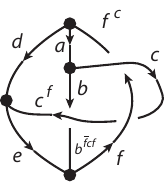}}$$
\caption{A knotted tetrahedron.}
\label{F:K4knot}
\end{figure}

\end{example}

\begin{example}
In \cite{BM}, Mellor and Backer Peral computed the $N$-quandle for the graph $G(k,m,n)$ shown in Figure~\ref{F:Gkmn}, with $N = (2,2,m,n,2,2)$. In particular, $\vert Q_N(G(k,m,n))\vert = 4kmn+2km+2kn$, where $\vert Q_N^1\vert = \vert Q_N^2\vert = \vert Q_N^5\vert = \vert Q_N^6\vert = kmn$, $\vert Q_N^3\vert = 2kn$ and $\vert Q_N^4\vert = 2km$. Therefore $\vert\pi_1^N(G(k,m,n))\vert = n_1\vert Q_N^1\vert = 2kmn$.

\begin{figure}[h]
$$\scalebox{.6}{\includegraphics{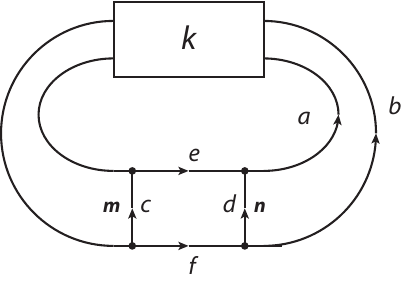}}$$
\caption{The spatial graph $G(k, m, n)$, where $k$ indicates the number of positive half-twists in the block.}
\label{F:Gkmn}
\end{figure}

\end{example}

%%%%%%%%%%%%%%%%%%%%%%%%%%%%%%%%%
%  Section: Links with finite $N$-quandles \label{S:linkfinite}
%%%%%%%%%%%%%%%%%%%%%%%%%%%%%%%%%
\section{Links and spatial graphs with finite $N$-quandles} \label{S:finite}

%\subsection{Links with finite $N$-quandles}\label{SS:linkfinite}

We can use Theorem \ref{T:QiffG} to give a complete list of the links with finite $N$-quandles. This extends the classification of links with finite $n$-quandles given by Hoste and Shanahan \cite{HS2}. Suppose $L$ is a link with $k$ components, $N = (n_1, \dots, n_k)$ is a $k$-tuple of positive integers, and $Q_N(L)$ is finite. By Theorem \ref{T:QiffG}, $\pi_1^N(L)$ is also finite.

Define $\mathcal{O}(L, N)$ to be the 3-orbifold with underlying space $\mathbb{S}^3$ and singular locus $L$, where the $i$th component of $L$ is labeled by $n_i$. (See \cite{CHK} for more information on orbifolds.) $\mathcal{O}(L,N)$ has a universal cover, and the group of covering transformations is the orbifold fundamental group $\pi_i^{orb}(\mathcal{O}(L,N))$. In this case, we have (by \cite[Theorem 2.9 ff.]{CHK})
$$\pi_i^{orb}(\mathcal{O}(L,N)) = \pi_1(\mathbb{S}^3 - L)/\langle \mu_i^{n_i} \rangle = \pi_1^N(L)$$
So the orbifold fundamental group is finite, which means the universal cover is a compact, simply-connected manifold. By Thurston's geometrization theorem, this means the universal cover is a sphere. Hence $\mathcal{O}(L, N)$ is a spherical 3-orbifold.  This proves:

\begin{theorem} \label{T:linkclass}
A link $L$ with $k$ components has a finite $(n_1,\dots, n_k)$-quandle if and only if there is a spherical orbifold with underlying space $\mathbb{S}^3$ whose singular locus is the link $L$, with component $i$ labeled $n_i$.
\end{theorem}

Unlike links, not every spatial graph is the singular locus of a 3-orbifold. For a graph to be a singular locus, it must be trivalent, and the 3 labels at each vertex must be $(2, 2, k)$ (where $k \geq 2$), $(2, 3, 3)$, $(2, 3, 4)$ or $(2, 3, 5)$ \cite[Theorem 2.5]{CHK}. But, in these cases, the same argument we used for Theorem \ref{T:linkclass} proves:

\begin{theorem} \label{T:graphclass}
Suppose a graph $G$ with $k$ edges, with edge $i$ labeled $n_i$, is the singular locus of a 3-orbifold.  Then $G$ has a finite $(n_1,\dots, n_k)$-quandle if and only if the orbifold is spherical, with underlying space $\mathbb{S}^3$.
\end{theorem}

Dunbar \cite{DU} classified all geometric, non-hyperbolic 3-orbifolds. He provided a list of all spherical 3-orbifolds with underlying space $\mathbb{S}^3$ and singular locus a link or spatial graph $L$; hence, it is also the list of all links, and all graphs which satisfy the conditions to be a singular locus of an orbifold, in $\mathbb{S}^3$ with finite $N$-quandle for some $N$. The sizes and structures of the $n$- and $N$-quandles of many of these links and graphs have been determined \cite{BM, CHMS, HS1, Me, MS}. The list is provided in the Appendix.

%%%%%%%%%%%%%%%%%%%%%%%%%%%%%%%%%
%  Section: Questions for further investigation
%%%%%%%%%%%%%%%%%%%%%%%%%%%%%%%%%
\section{Questions for further investigation} \label{S:questions}

Finally, we pose a few questions for future study. While we have completely classified the links with finite $N$-quandles, we have not done the same for spatial graphs. 

\begin{question}
Are there spatial graphs with finite $N$-quandles which are {\em not} the singular locus of a spherical 3-orbifold?
\end{question}

One approach might be to explore how various operations on spatial graphs affect the fundamental quandle. We know that deleting edges from a graph with a finite $N$-quandle, or splitting an edge with a vertex of valence two, yield a new graph with a finite $N$-quandle \cite{BM}.  But there are many other graph operations (such as contracting edges) that could be explored.

\medskip

Even among the links and spatial graphs that are known to have finite $N$-quandles, the precise size and structure are not all known. In particular, in Table \ref{linktable} in the Appendix, the 2-quandles of the links in the last row, and the $N$-quandles for the graphs created by adding struts to the rational tangles of the links in the last two rows, have not been completely described.  

\begin{question}
What are the sizes and structures for the finite $N$-quandles that are not described in \cite{BM, CHMS, HS1, Me, MS}?
\end{question}

Finally, the fundamental quandle is a complete invariant for (unoriented) knots and (unoriented, non-split) links; but it is not known how powerful it is for spatial graphs. As we saw in Section \ref{S:fundamental}, the fundamental quandle of a knot or link immediately determines the fundamental group.  For a spatial graph, on the other hand, to recover the fundamental group from the fundamental quandle you also need to know the vertex relations, which depend on the particular diagram. So the fundamental quandle on its own is likely not a complete invariant; but what other information is needed to construct a complete invariant?

\begin{question}
To what extent does the fundamental quandle of a spatial graph determine the spatial graph? What other information is needed to give a complete invariant for unoriented, connected spatial graphs?
\end{question}

\bibliographystyle{plain}
\bibliography{Nquandle.bib}

\newpage
%%%%%%%%%%%%%%%%%%%%%%%%%%%%%%%%%
%  Appendix: links and graphs with finite N-quandles
%%%%%%%%%%%%%%%%%%%%%%%%%%%%%%%%%
\section*{Appendix: Links and graphs with finite $N$-quandles} \label{S:Appendix}

\begin{table}[htbp]
%\tbl{Links $L \in \mathbb{S}^3$ with finite $Q_n(L)$.}
{$
\begin{array}{cccc}
\includegraphics[width=1.25in,trim=0 0 0 0,clip]{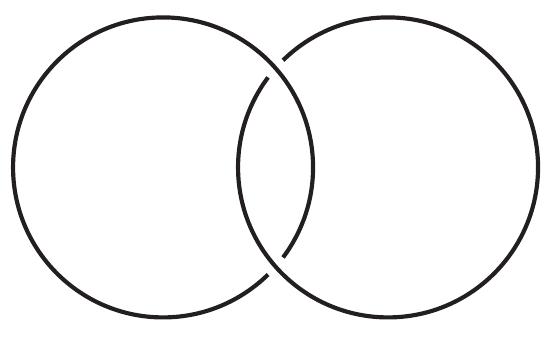}   & \includegraphics[width=1.25in,trim=0 0 0 0,clip]{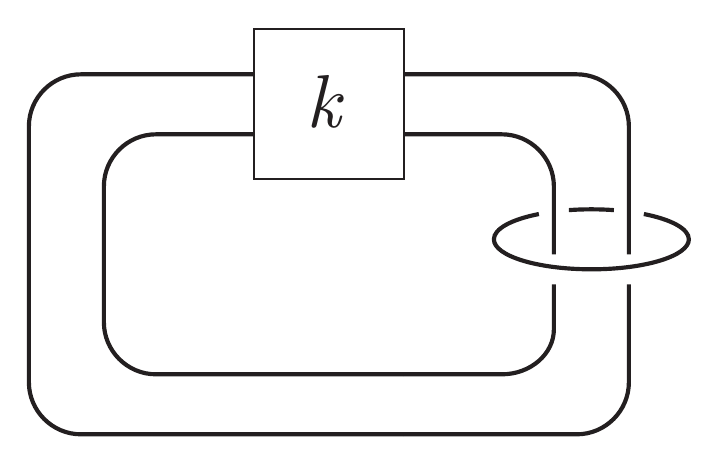}  & \includegraphics[width=1.0in,trim=0 0 0 0,clip]{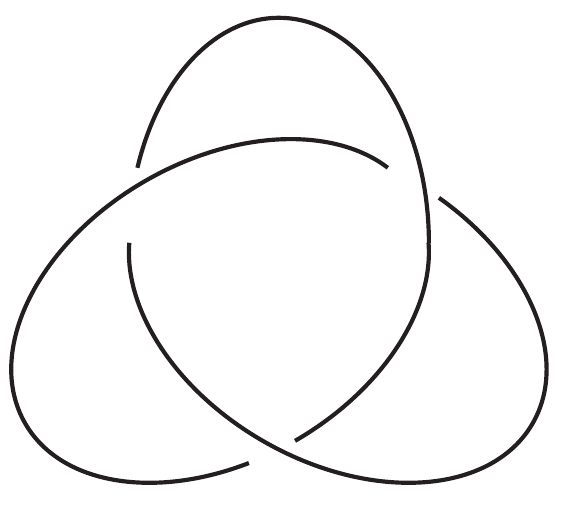} \\
\scriptstyle n > 1 & \scriptstyle k \neq 0,\  n=2& \scriptstyle n=3, 4, 5 \\
\\
\includegraphics[width=1.0in,trim=0 0 0 0,clip]{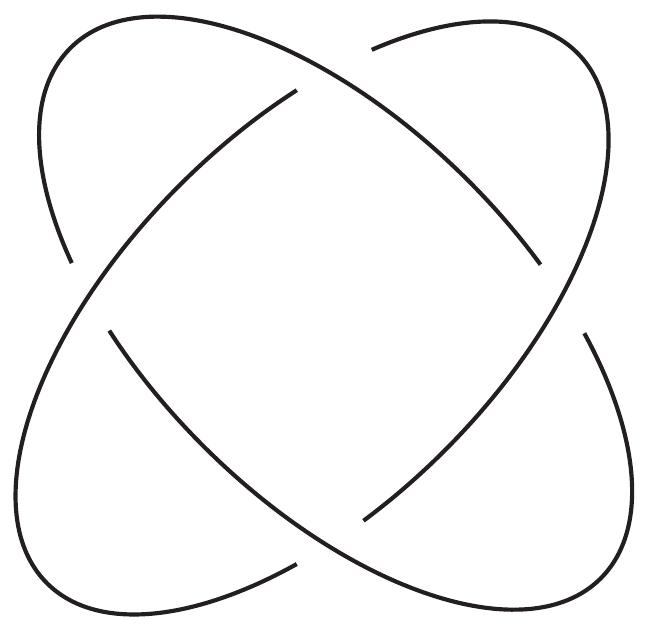}  & \includegraphics[width=1.0in,trim=0 0 0 0,clip]{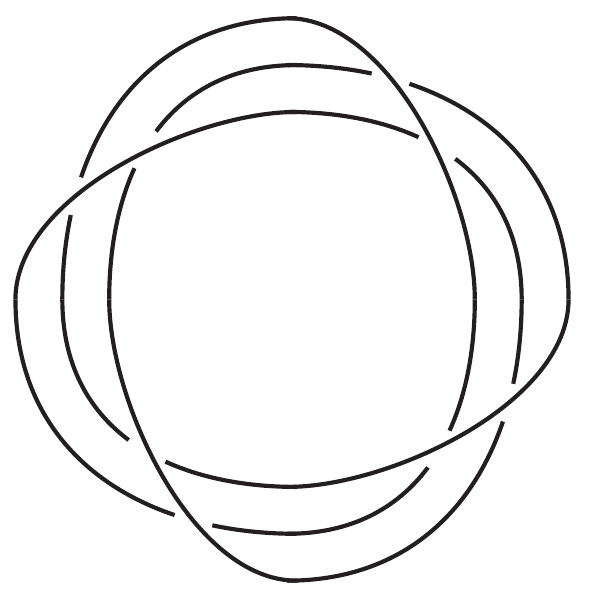}  & \includegraphics[width=1.0in,trim=0 0 0 0,clip]{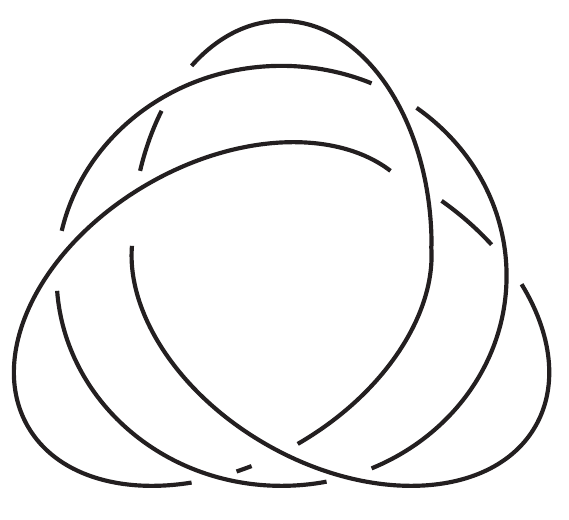} \\
\scriptstyle n =3 & \scriptstyle n=2& \scriptstyle n=2 \\
\\
\includegraphics[width=1.0in,trim=0 0 0 0,clip]{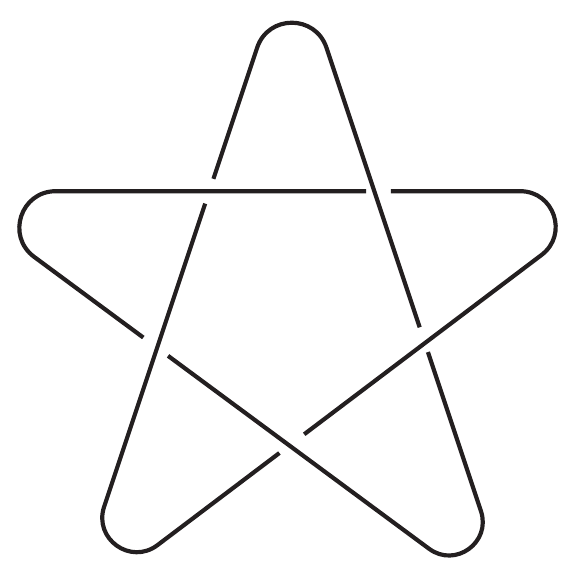}  & \includegraphics[width=1.0in,trim=0 0 0 0,clip]{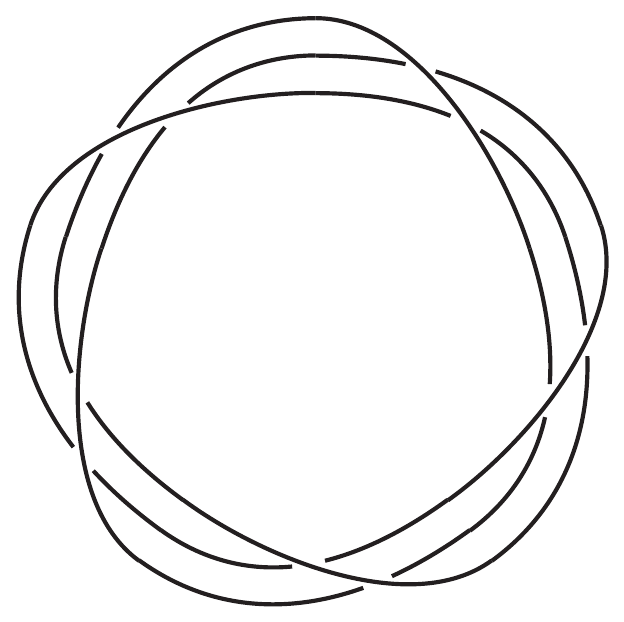}  & \includegraphics[width=1.25in,trim=0 0 0 0,clip]{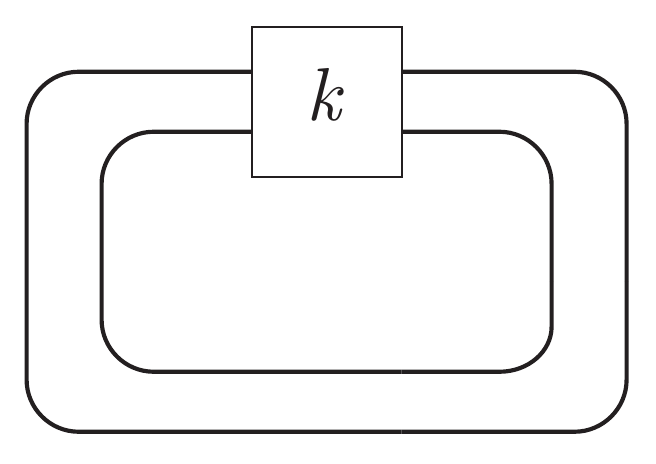} \\
\scriptstyle n =3 & \scriptstyle n=2& \scriptstyle k\neq0,\ n=2 \\
\\
\includegraphics[width=1.25in,trim=0 0 0 0,clip]{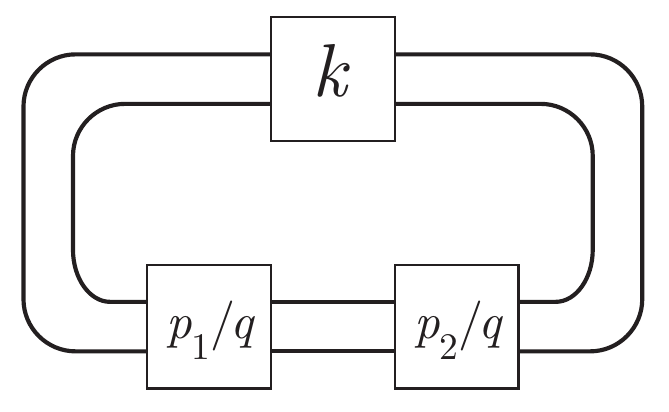}  & \includegraphics[width=1.15in,trim=0 5pt 0 0,clip]{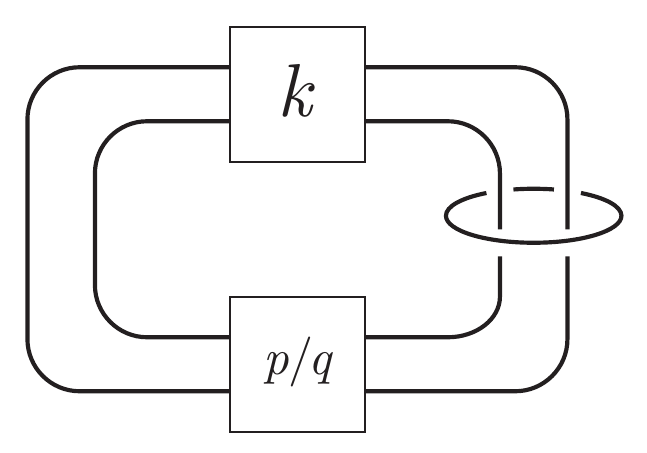} &  \includegraphics[width=1.65in,trim=0 0 0 0,clip]{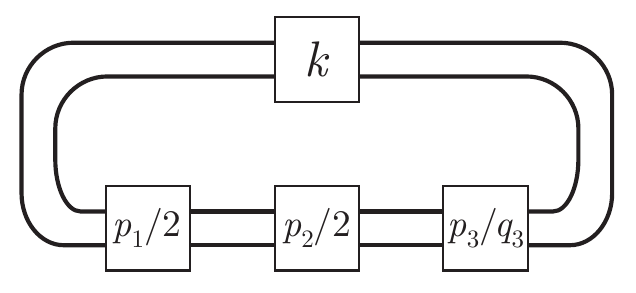} \\
\scriptstyle k+p_1/q+p_2/q \neq 0,\ n =2  &\scriptstyle n=2& \scriptstyle k+p_1/2+p_2/2+p_3/q_3 \neq 0,\ n =2 \\
\\
\includegraphics[width=1.65in,trim=0pt 0pt 0pt 0pt,clip]{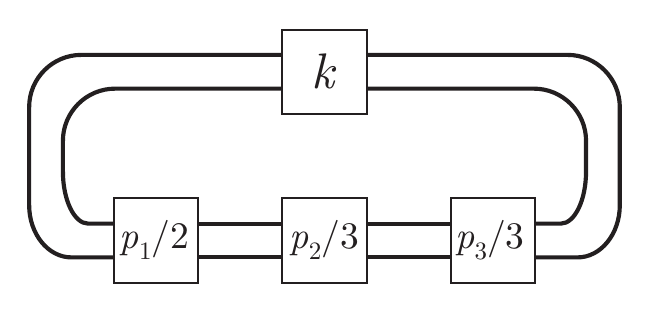}  & \includegraphics[width=1.65in,trim=0pt 0pt 0pt 0pt,clip]{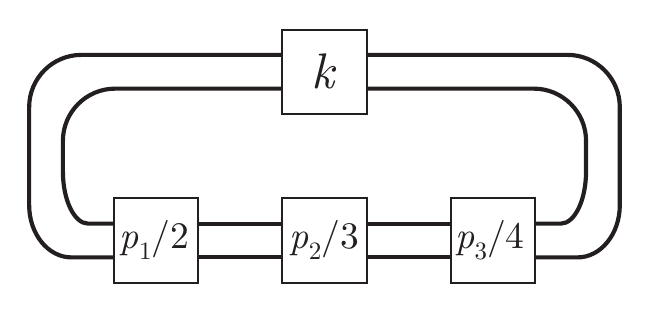}  & \includegraphics[width=1.65in,trim=0pt 0pt 0pt 0pt,clip]{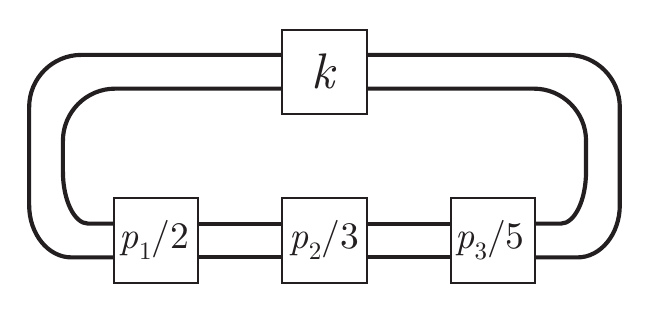}\\
\scriptstyle k+p_1/2+p_2/3+p_3/3 \neq 0,\ n =2  & \scriptstyle k+p_1/2+p_2/3+p_3/4 \neq 0,\ n =2  & \scriptstyle k+p_1/2+p_2/3+p_3/5 \neq 0,\ n =2
\end{array}
$}
\bigskip

\caption{\parbox{3.75in}{Links $L \in \mathbb{S}^3$ with finite $Q_n(L)$. Here \fbox{$k$} represents $k$ right-handed half-twists, and \fbox{$p/q$} represents a rational tangle. If $p$ and $q$ are not relatively prime, the tangle contains a ``strut" labeled $\gcd(p,q)$ \cite{DU}, and the resulting spatial graph has a finite $N$-quandle.}}
\label{linktable}
\end{table}

\newpage

\begin{table}[htbp]
%\tbl{Links $L \in \mathbb{S}^3$ with finite $Q_N(L)$.}
{$
\begin{array}{ccc}
\includegraphics[width=1.5in,trim=0 0 0 0,clip]{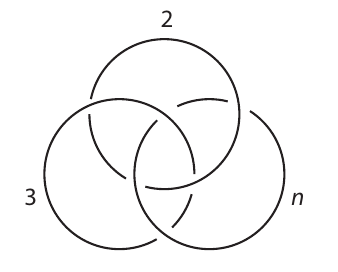}   & \hspace{.5in}& \includegraphics[width=1.25in,trim=0 0 0 0,clip]{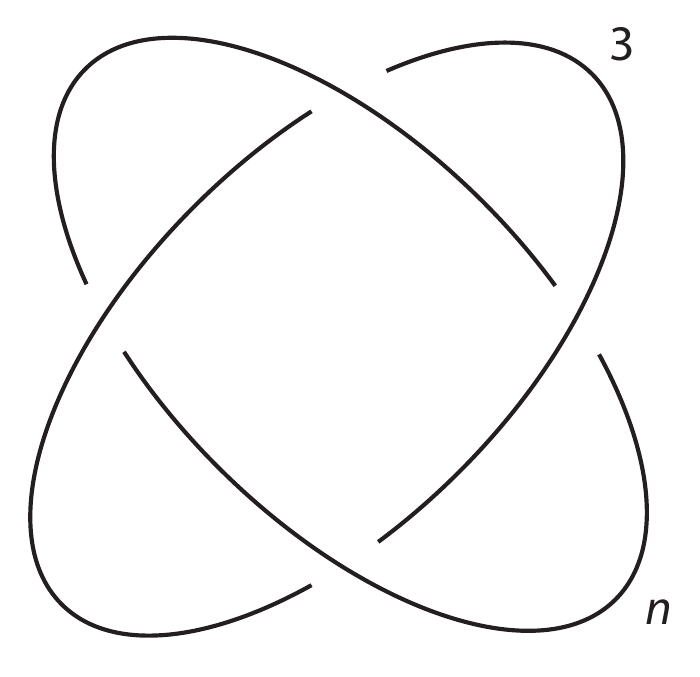}  \\
\scriptstyle L = T_{3,3};\ N = (2, 3, n);\ n = 3, 4, 5 && \scriptstyle L = T_{2,4}; N = (3, n);\ n=3, 4, 5 \\
\\
\includegraphics[width=1.25in,trim=0 0 0 0,clip]{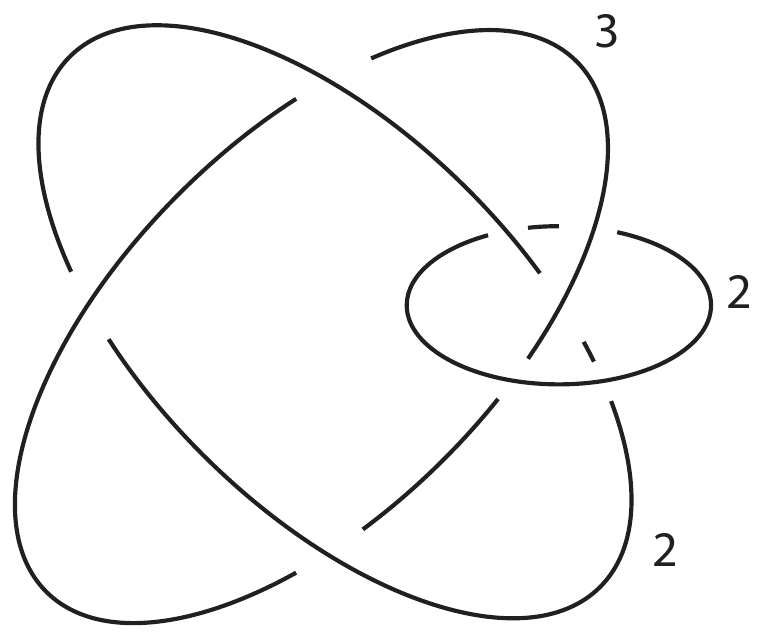} && \includegraphics[width=1.5in,trim=0 0 0 0,clip]{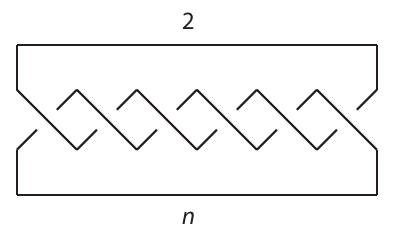}  \\
\scriptstyle L = T_{2,4} \cup C; N=(2,2,3) && \scriptstyle L = T_{2,6}; N = (2, n);\ n=3, 4, 5 \\
\\
 \includegraphics[width=1.75in,trim=0 0 0 0,clip]{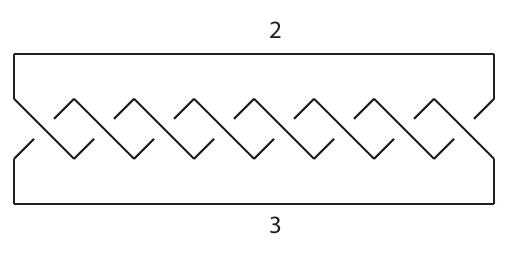}  && \includegraphics[width=2in,trim=0 0 0 0,clip]{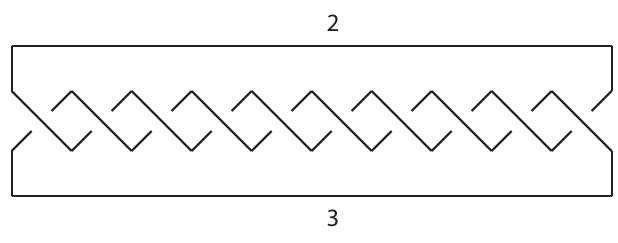} \\
\scriptstyle L = T_{2,8}; N = (2, 3) && \scriptstyle L = T_{2,10}; N = (2, 3) \\
\includegraphics[width=1.5in,trim=0 0 0 0,clip]{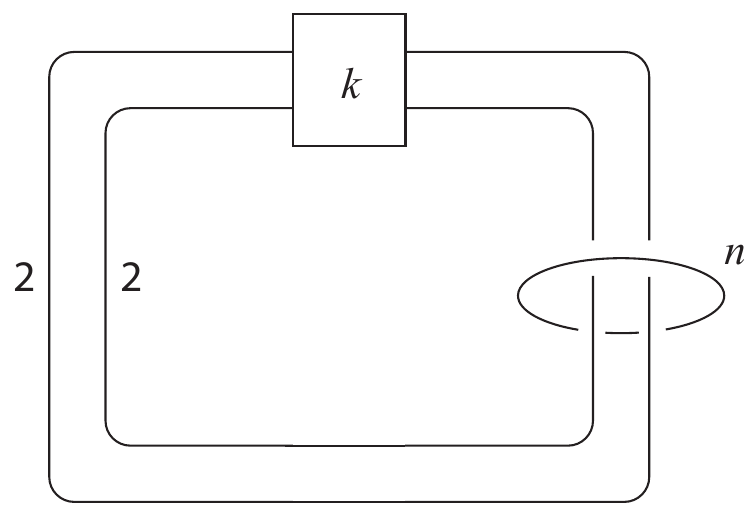}  && \includegraphics[width=1.5in,trim=0 0 0 0,clip]{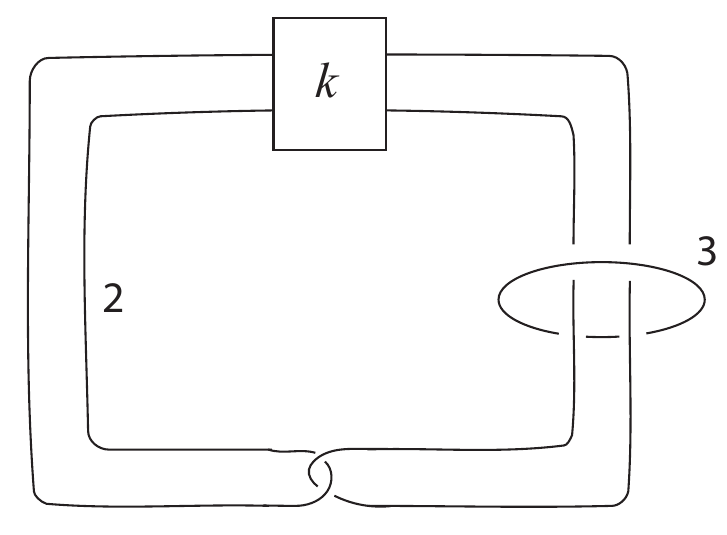}  \\
\scriptstyle L_k = T_{2,k} \cup C; N = (2, n)\ {\rm or}\ (2, 2, n);\ n > 1;\ k \neq 0 && M_k = T_k \cup C; \scriptstyle N = (2, 3)
\end{array}
$}
\bigskip

\caption{Other links $L \in \mathbb{S}^3$ with finite $Q_N(L)$.}
\label{linktable2}
\end{table}

\begin{table}[htbp]
\begin{center}
{$
\begin{array}{ccc}
\includegraphics[height=1in,trim=0 0 0 0,clip]{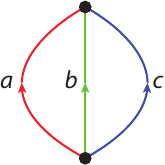} & \hspace{.5in} & \includegraphics[height=1in,trim=0 0 0 0,clip]{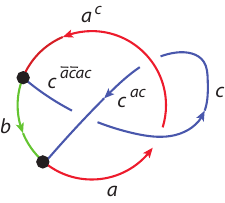} \\
\scriptstyle \text{Theta graph } \theta_3 & & \scriptstyle \text{Knotted theta graph } KT \\
\scriptstyle N = (2,2,2), (3,2,2), (n, 3, 2); n = 3, 4, 5 & & \scriptstyle N = (3, 3, 2) \\
&& \\
\end{array}
$}
{$
\begin{array}{ccc}
\includegraphics[height=.8in,trim=0 0 0 0,clip]{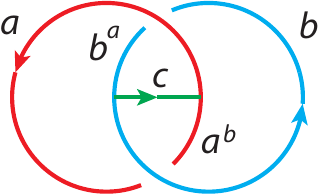} & \hspace{.5in} & \includegraphics[height=1in,trim=0 0 0 0,clip]{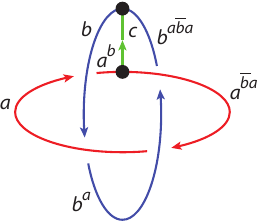}  \\
\scriptstyle \text{Hopf Handcuff graph } H_1 & & \scriptstyle \text{2-linked Handcuff graph } H_2  \\
\scriptstyle N = (3, 2, 2), (3, 3, 2) & & \scriptstyle N = (3, 2, 2)  \\
&& \\
\end{array}
$}
{$
\begin{array}{ccc}
\includegraphics[height=.8in,trim=0 0 0 0,clip]{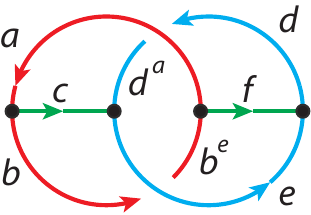} & \hspace{.5in} & \includegraphics[height=1in,trim=0 0 0 0,clip]{K4Knot.pdf}  \\
\scriptstyle \text{Double Handcuff graph } DH & & \scriptstyle \text{Knotted } K_4  \\
\scriptstyle N = (2,2,2,3,2,2), (2,2,3,3,2,2), (2,2,2,3,2,4) & & \scriptstyle N = (3,3,2,2,2,2) \\
&& \\
\end{array}
$}
{$
\begin{array}{ccc}
& \includegraphics[height=1in,trim=0 0 0 0,clip]{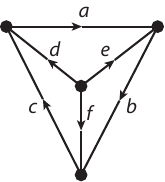} &  \\
& \scriptstyle \text{Planar } K_4 & \\
& \scriptstyle N = (3, n, 2,2,2,2), (3, 3, 2, 2, 2, n); n = 2, 3, 4, 5 & \\
& \scriptstyle N = (3,3,3,2,2,2), (3,4,2,2,2,3) & \\
&& \\
\end{array}
$}
\end{center}
\bigskip

\caption{Other graphs with finite $N$-quandles.}
\label{T:exceptional}
\end{table}

%\clearpage
%\section*{References}

\end{document}